\documentclass{article}

\usepackage{amsmath}
\usepackage{amssymb}
\usepackage{hyperref}
\usepackage{savesym}
\savesymbol{AND}
\usepackage{algorithm}
\usepackage[noend]{algorithmic}
\usepackage{enumitem}
\usepackage[all]{xy}

\newtheorem{theorem}{Theorem}[section]
\newtheorem{df}[theorem]{Definition}
\newtheorem{rem}[theorem]{Remark}

\newtheorem{prop}[theorem]{Proposition}
\newtheorem{lem}[theorem]{Lemma}
\newtheorem{example}[theorem]{Example}%
\newtheorem{corollary}[theorem]{Corollary}
\newtheorem{observation}[theorem]{Observation}

\newenvironment{proof}{\noindent
                       {\bf Proof:}}{
                       \begin{flushright}$\Box$\end{flushright}}

\title{On discriminants, Tjurina modifications and the geometry of 
determinantal singularities}

\author{Anne Fr\"uhbis--Kr\"uger\footnote{partially supported by DFG priority program SPP 1489 'Algorithmic and Experimental Methods in Algebra, Geometry and Number Theory' and by NTH-grant 'Experimental Methods in Computer Algebra'}\\Institut f. Alg. Geometrie,Leibniz Universt\"at Hannover, Germany}

\begin{document}

\maketitle

\begin{abstract}
We describe a method for computing discriminants for a large 
class of families of isolated determinantal singularities -- more precisely, 
for subfamilies of ${\mathcal G}$-versal families. The approach intrinsically 
provides a decomposition of the discriminant into two parts and allows the 
computation of the determinantal and the non-determinantal loci of the family 
without extra effort; only the latter manifests itself in the Tjurina 
transform. This knowledge is then applied to the case of Cohen-Macaulay 
codimension 2 singularities putting several known, but previously unexplained 
observations into context and explicitly constructing a counterexample to
Wahl's conjecture on the relation of Milnor and Tjurina numbers for surface 
singularities.
\end{abstract}

\section{Introduction}
Isolated hypersurface and complete intersection singularities are well studied 
objects and there are many classical results about different ascpects such as 
topology, deformation behaviour, invariants, classification and even metric 
properties (see any textbook on singularities, e.g. \cite{Loo}, \cite{GLS}, 
\cite{Ish}). Beyond complete intersections, however, knowledge is rather 
scarce and unexpected phenomena arise. In this artcle, we focus on the 
class of determinantal singularities to pass beyond ICIS, as the properties 
already differ significantly, but classical results on determinantal
varieties and free resolutions provide strong tools to treat this case. 
Recently, significant progress has been made for this class, e.g. in  
\cite{PeR}, \cite{BCR}, \cite{NOT}, \cite{DP}, \cite{GR}. In \cite{FZ}
the use of Tjurina modifications made it possible to relate a given 
determinantal singularity to an often singular variety, which happens to be
an ICIS under rather mild conditions. This could e.g. be exploited in \cite{FZ}
and \cite{Zac} to determine the topology of the Milnor fibre of an isolated
Cohen-Macaulay codimension $2$ singularity (ICMC2 singularity for short).
But it is also obvious from those results that the Tjurina transform is blind
to certain other properties of an ICMC2 singularity. In this article, we 
explain which properties of the singularity manifest themselves in the 
Tjurina transform and which do not by studying the discriminant and a natural
decomposition thereof. \\
The general approach to determining the discriminant of a given family of 
varieties $V(I_{\underline{t}})$ is based on the Jacobian criterion. It 
involves the elimination of the original variables from the ideal generated 
by $I_{\underline{t}}$ and the ideal of minors of appropriate size of the 
relative Jacobian matrix of $I_{\underline{t}}$. However, the complexity 
of this approach, which originates from the sensitivity of Gr\"obner basis 
computations to the number of occurring variables, makes it impractical for
many examples. It is hence important to understand the structure of the
discriminant theoretically and to be able to decompose it appropriately 
by a priori arguments.
Making use of Hironaka's smoothness criterion \cite{Hir}, the structure 
of the perturbed matrix can be used to split the problem into two
smaller problems, one dealing with the locus of determinantal singularities,
the other one closely related to the Tjurina transform as it describes the
locus above which there are singularities adjacent to an $A_1$ singularity. \\

In section \ref{basic}, we first recall known facts about determinantal
singularities and then proceed to revisit Hironaka's smoothness criterion. 
In the following section \ref{discr}, we consider the discriminant of versal families of 
determinantal singularities of type $(m,n,t)$ starting with the simplest case
$(2,k,2)$, then passing on to maximal minors of matrices of arbitrary size 
and finally to smaller minors. As a sideeffect, we also obtain a quite explicit
formulation of the Tjurina transform in the non-maximal case.  The 
above mentioned decomposition of the discriminant into the two parts also
gives rise to the surprising behaviour of some aspects of ICMC2 
singularities, as we are seeing an interplay of influences related to 
properties of the generic determinantal singularity and to the Tjurina 
transform. With these two contributions in mind, it is possible to predict some
properties of the singular locus of the Tjurina transform, to explain
observations of \cite{DP} about ICMC2 $3$-folds and prove the easy direction 
of Wahl's conjecture \cite{Wah} on the relation between Milnor and Tjurina 
number for ICMC2 surface singuarities. The knowledge from this proof then
leads to the construction of a class of counter examples for the converse 
direction of the conjecture. These applications to the ICMC2 case are 
discussed in the final section.\\

The author wishes to thank Matthias Zach, Maria Ruas, Terry Gaffney, 
James Damon, Juan Nu\~no-Ballesteros and David Mond for many fruitful 
discussions on determinantal singularities. Moreover, she wishes to express 
her gratitude to the organizers of the International Workshop on Real and 
Complex Singularities in S\~{a}o Carlos in 2016, who brought together all of 
these colleagues in the same place at the same time and created a 
mathematically stimulating atmosphere, in which some of the questions adressed
in this article came up. The author is also endepted to Gert-Martin Greuel for
his remarks leading to significant improvements of the exposition. All examples
listed in this article were computed in {\sc Singular} \cite{Sing}. \\

\section{Basic Facts on EIDS} \label{basic}

Before focussing on the discriminant, we shall first recall the definition 
of properties of the class of singularities on which we focus in the following 
sections: essentially isolated determinantal singularities (EIDS) as  first 
introduced by Ebeling and Gusein-Zade in \cite{EGZ}. Although certain classes 
of such singularities had been studied before (e.g. \cite{FK1}, \cite{FN} and
\cite{BT}), simultaneous studies of certain properties of EIDS of all 
matrix sizes and types only appear recently e.g. in 
\cite{BCR} and \cite{NOT}.
In this section, we cover well-known facts about EIDS to give the reader the 
background knowledge for the subsequent considerations on the discriminant 
and the Tjurina transoform.

\begin{df}
Let $M_{m,n}$ denote the set of all $m \times n$-matrices with entries in 
${\mathbb C}$ and let $1 \leq t \leq {\rm min}\{m,n\}$. Then
$$M_{m,n}^t : = \{A \in M_{m,n} \mid rk(A) < t\}$$
is the generic determinantal variety. 
\end{df}

\begin{rem} \label{genMatrix}
$M_{m,n}^t$ can be understood as the variety in 
$M_{m,n} \cong {\mathbb C}^{mn}$ which is the vanishing locus of the 
ideal of $t$-minors of the matrix
$$\begin{pmatrix} x_1 & \dots & x_n \cr
                 \vdots & & \vdots \cr
                 x_{(m-1)n+1} & \dots & x_{mn} \end{pmatrix} \in
   \operatorname{Mat}(m,n,{\mathbb C}[x_1,\dots,x_{mn}]).$$
$M_{m,n}^t$ has codimension $(n-t+1)(m-t+1)$ in $M_{m,n}$ and 
$\operatorname{Sing}(M_{m,n}^t)=M_{m,n}^{t-1}$. Moreover, it is known that
the sets $M_{m,n}^i \setminus M_{m,n}^{i-1}$ for $1 \leq i \leq t$ form a 
Whitney stratification for $M_{m,n}^t$.\footnote{Passing to the convergent
power series ring, we can also define and use these notions in the setting 
of analytic space germs. Analogously all of the subsequent notions can be 
carried over to space germs and the analytic setting.}   
\end{rem}

\begin{rem}
By a result of Ma \cite{Ma93} extending several previous results for maximal, 
for submaximal and for $2$-minors (see e.g. \cite{Bu79}, \cite{Akin82}), 
the first syzygy module of the ideal of $M_{m,n}^t$ is generated by linear 
relations.
\end{rem}

As we are interested in minors of matrices with arbitrary power series as 
entries, the generic determinantal varieties are not the objects of our 
primary focus. They are only a tool to formulate the following definition:

\begin{df}
Let $F: {\mathbb C}^k \longrightarrow M_{m,n} \cong {\mathbb C}^{mn}$ be a 
polynomial map. Then $X=F^{-1}(M_{m,n}^t)$ is a determinantal variety of 
type $(m,n,t)$, if $\operatorname{codim}(X) = (m-t+1)(n-t+1)$ in 
${\mathbb C}^k$. 
\end{df}

\begin{rem} \label{linSyzF}
By a well-known result of Eagon and Hochster \cite{EH}, the condition on the 
codimension ensures that the local ring of a germ 
of a determinantal variety is Cohen-Macaulay. \\
The same condition on the codimension also implies that all syzygies 
arise from the linear ones of the generic matrix (by the extension of 
scalars via the map of local rings induced by $F$). As a consequence, 
the entries of the syzygy matrix are linear combinations of the
components $F_{i,j}$ of the map $F$.
\end{rem}

\begin{df}[\cite{EGZ}]
A germ $(X,0) \subset ({\mathbb C}^k,0)$ of a determinantal variety of type 
$(m,n,t)$ is called an EIDS at $0$, if the corresponding $F$ is transverse 
to all strata $M_{m,n}^i \setminus M_{m,n}^{i-1}$ of $M_{m,n}^t$ outside the 
origin.
\end{df}

\begin{rem}
For an EIDS $(X,0)$ defined by $F^{-1}(M_{m,n}^t)$ the singular locus is 
precisely the preimage of the singular locus of $M_{m,n}^t$, i.e. it is 
$F^{-1}(M_{m,n}^{t-1})$. Thus $(X,0)$ has an isolated singularity at the origin
iff $k \leq (m-t+2)(n-t+2)$; moreover, a smoothing of it exists if and only if 
the previous inequality is strict.
\end{rem}

As invertible row and column operations applied to a matrix do not change 
the ideal of its minors, this holds for any matrix of an EIDS. Moreover, two 
singularities should be considered equivalent, if one arises from the other 
by means of a coordinate change. These observations explain the structure 
of the group which describes the most suitable equivalence relation 
for determinantal singularities:

\begin{df}
Let ${\mathcal G}$ denote the group $(Gl_m({\mathbb C}\{\underline{x}\}) 
\times Gl_n({\mathbb C}\{\underline{x}\})) 
\rtimes \operatorname{Aut}({\mathbb C}\{\underline{x}\})$. Two determinantal 
singularities 
$(X_1,0), (X_2,0) \subset ({\mathbb C}^k,0)$ of the same type $(m,n,t)$ 
and defined by $F_1$ and $F_2$ respectively are called 
${\mathcal G}$-equivalent, if there is a tuple $(R,L,\Phi) \in {\mathcal G}$
such that $F_1 = L^{-1} (\Phi^* F_2)R$.  
\end{df}

Recall that a map germ is $s$-determined, if any other map which coincides 
with it in all terms up to degree $s$ is equivalent to it. If we want to
stress the existence of such an $s$ without specifying the value of $s$, the 
map is referred to as finitely determined.

\begin{theorem}\cite{Pthesis}
An EIDS $(X,0)$ corresponding to a map $F$, defined by the ideal of minors 
of the corresponding matrix (which we also denote by $F$ by abuse of notation),
is finitely ${\mathcal G}$-determined if and only if it has finite
${\mathcal G}$-Tjurina number 
$$\tau_{\mathcal G} = \operatorname{dim}_{\mathbb C}
     (\operatorname{Mat}(m,n;{\mathbb C}\{\underline{x}\}) / (J_F + J_{op}))$$
where $J_F$ denotes the submodule generated by the $k$ matrices, each holding 
the partial derivatives of the entries of $F$ w.r.t. one of the variables, and
$$J_{op}= \langle AF + FB \mid A\in \operatorname{Mat}(m,m;{\mathbb C}\{\underline{x}\}),
 \operatorname{Mat}(n,n;{\mathbb C}\{\underline{x}\}) \rangle$$
\end{theorem}

\begin{rem}
The group ${\mathcal G}$ is a subgroup of the group ${\mathcal K}$ of Mather 
and coincides with it e.g. for complete intersections and Cohen-Macaulay 
codimension $2$ singularities. It also appears as the subgroup ${\mathcal K}_V$
of ${\mathcal K}$ in the literature, where $V$ is the generic determinantal 
singularity of the appropriate type (defined by $M_{m,n}^t$). In the cases, 
where ${\mathcal G}$-equivalence coincides with analytic equivalence, 
$\tau_{\mathcal G}$ is precisely the usual Tjurina number. In general, however,
${\mathcal G}$ is a proper subgroup of ${\mathcal K}$
as it respects the underlying matrix size. Thus there can e.g. 
be no element of ${\mathcal G}$ leading from one representation of Pinkham's 
famous example \cite{Pin} to the other, i.e. leading from a determinantal 
singularity of type $(2,4,2)$ to one of type $(3,3,2)$ (with the additional 
constraint of the matrix to be symmetric) or vice versa.
It is important to observe that a restriction to ${\mathcal G}$-equivalence 
does not fix the minimal size of the matrix, it only fixes some size, as any
determinantal singularity of type $(m,n,t)$ can easily be considered as one
of type $(m+1,n+1,t+1)$ by simply adding an extra line and an extra column 
of which all entries are zero except the one where the row and column meet, 
which should then be chosen to be $1$. If, on the other hand, a 
${\mathcal G}$-equivalence class of $(m+1) \times (n+1)$-matrices contains a matrix of this particular structure , the class will be referred to as 
essentially of type $(m,n,t)$.
\end{rem}

Keeping in mind, that the chosen equivalence can only be used to compare
determinantal singularities of compatible types, we shall restrict our 
considerations to families for which each fibre is a determinantal singularity 
of appropriate type. A similar restriction has already been used by Schaps
in \cite{Sch},
where she considered the notion of an $M$-deformation (or more generally an
$f$-deformation) by deforming the entries of a given matrix $M$. There she 
gives a criterion when an $M$-deformation is versal and provides examples of
situations, in which it is not. However, the $M$-deformation of Schaps does
not even provide a versal unfolding of the $F$ as her example 1 shows, which 
she attributes to D.~S.~Rim without further reference:

\begin{example}
Consider the determinantal singularity given by the $2$-minors of the matrix
$$\begin{pmatrix}
  x_1 & \alpha x_2 & \beta x_3 & \gamma x_4 \cr
  x_1 &        x_2 &       x_3 &        x_4
 \end{pmatrix} \sim_{\mathcal G}
 \begin{pmatrix}
  x_1 &          0 &     s x_3 &      t x_4 \cr
    0 &        x_2 &       x_3 &        x_4
 \end{pmatrix},$$
where $\alpha, \beta, \gamma \in {\mathbb C}$ sufficiently general (according
to Schaps) or more precisely $(s:t) \in {\mathbb P}^1 \setminus \{(0:1),(1:0),
(1:1)\}$, which immediately allows us to pass to the affine chart $t\neq 0$
writing $S$ for $\frac{s}{t}$.
Note that two such matrices corresponding to different points in 
${\mathbb P}^1 \setminus \{(0:1),(1:0),(1:1)\}$ lead to 
non-${\mathcal G}$-equivalent matrices, as this would alter the cross-ratio
of the four points $(0:1)$, $(1:0)$, $(1:1)$, $(s:t)$ in ${\mathbb P}^1$.
The corresponding space germs, on the other hand,  are isomorphic in a 
trivial way, as a change of $e$ only manifests itself by multiplying some
of the minors by invertible constants.\\
$\tau_{\mathcal G}$ is $5$ in this example and a versal
unfolding of the corresponding morphism $F$ is given by
$$\begin{pmatrix}
  x_1 &   a & Sx_3 + b & x_4 + c + e x_4 \cr
    d & x_2 & x_3      & x_4
  \end{pmatrix},$$
where Schaps only accepts the deformation parameters $a,b,c,d$, but refuses 
$e$, as it alters the original matrix. An interesting property of this family
is that changes to the parameter $e$ only have an effect on the 
${\mathcal K}$-equivalence class of the corresponding space germ, if at 
least one of the other four parameters is non-zero.
\end{example}

Choosing a monomial \footnote{Following tradition in the standard basis 
community we also refer to a module element, of which the only non-zero entry
is a monomial, as a monomial.} ${\mathbb C}$-basis $m_1,\dots,m_{\tau}$ of the 
$\operatorname{Mat}(m,n;{\mathbb C}\{\underline{x}\}) / (J_F + J_{op})$, it is then easy to
write down a semiuniversal unfolding of the morphism $F$, by simply perturbing the
corresponding matrix as follows:
$$F_{t_1,\dots,t_{\tau}} = F + \sum_{i=1}^{\tau} t_im_i.$$
The corresponding family of space germs is also versal for determinantal 
deformations of determinantal singularities\footnote{For simplicity of 
notation, we can consider a non-singular codimension $k$ germ as 
determinantal singularity essentially of type $(1,k,1)$.} in the following 
sense: Any family of space germs with given 
determinantal singularity in the zero-fibre and only determinantal 
singularities of appropriate type as fibres can be induced from the family of
space germs described by $F_{t_1,\dots,t_{\tau}}$. To make notation a bit 
shorter, we call such a family of space germs ${\mathcal G}$-versal. \\

But the relation between the base of a ${\mathcal G}$-versal family and a
versal family can be quite subtle and has not yet been studied in 
generality. 
In Pinkham's example, we only see one of the components of the 
base of the versal family as the base of the ${\mathcal G}$-versal family; 
in Rim's example cited above the relation is significantly less obvious:

\begin{example} (2.11 continued)
A ${\mathcal G}$-versal family with base $({\mathbb C}^5,0)$ has already been 
constructed above. \\
A versal family of space germs with the given special fibre is
\begin{eqnarray*}
I_{\mathcal X}& = & \langle (S-1)x_3x_4+Ax_4+Dx_3,x_2x_4-Bx_4-Gx_2,
                            x_1x_4+Cx_4-CG,\\
 & & \;\;Sx_2x_3+Ex_3+Hx_2,x_1x_3-Fx_3-\frac{1}{S}FH,x_1x_2+\frac{1}{S}EF\rangle
\end{eqnarray*}
over a base $(V,0)$ which is the germ of the cone of the Segre embedding of 
${\mathbb P}^3 \times {\mathbb P}^1$ in ${\mathbb P}^7$. More precisely,
\begin{eqnarray*}
I_{V} & = & \langle SAB+AE-(S-1)BH, SAC+SAF-(S-1)FH, SBC-EF,\\
 & & \;\; SBD+DE+(S-1)EG,CD+DF+(S-1)CG, \\
 & & \;\;SAG-DH-(S-1)GH \rangle
\end{eqnarray*}
(see example 3.4 in the Th\`ese of Buchweitz \cite{Bu81} for the construction).
The relation between these seemingly completely unrelated base spaces becomes
apparent, as soon as we consider them as the images of $(V(B),0)$ under
the projections to the first and second factor of
$({\mathbb C}^5,0) \times ({\mathbb C}^8,0)$, where
\begin{eqnarray*}
B & = & \langle (S-1)b-(S-1-e)A, a-(1+e)B,(1+e)d+C,\\
 & & \;\; (S-1)c+(S-1-e)D,a+E,Sd-F,c+(1+e)G,b-H \rangle.
\end{eqnarray*}
(Note that the coefficients $S$, $S-1$, $S-1-e$ and $1+e$ are all units in 
the local ring).   
\end{example} 

The non-trivial comparison illustrated above also makes the question of 
semiuniversality difficult to answer in generality; in some cases as, e.g. the 
Hilbert-Burch case, it is obvious, in others it may boil down to a case
by case check for different matrix structures. In view of the possibility of
semiuniversality in further cases, we continue with the considerations in 
full generality and leave this question to be answered at the time of
application to specific matrix sizes (or even matrices). 

At this point, it is important to observe that the ${\mathcal G}$-versal 
families obtained by the above construction are indeed flat, as any relation 
lifts to a relation of the family by the second part of remark \ref{linSyzF}.\\
Having restricted our interest to ${\mathcal G}$-versal deformations, we can
now state the object, which we want to determine algorithmically in the 
subsequent section: the locus in the base ${\mathbb C}^{\tau_{\mathcal G}}$ 
of the ${\mathcal G}$-versal family, above which the fibres possess 
singularities, i.e. the ${\mathcal G}$-discriminant locus of the family.  \\

For later use, we also need one further construction concerning determinantal
singularities: a Tjurina modification as introduced in \cite{Tju} and used
e.g. in \cite{vSt} and \cite{FZ}. This construction relies on the fact that 
the rows of an $m \times n$ matrix $A$ representing a point of $M_{m,n}^t$
span a $(n-t+1)$-dimensional subspace in ${\mathbb C}^n$. This gives rise 
to a rational map 
$P: M_{m,n}^t \dashrightarrow \operatorname{Grass}(n-t+1,n)$, 
of which we can resolve indeterminacies to obtain a map 
$\hat P: W =\overline{ \Gamma_{P}(M_{m,n}^t \setminus M_{m,n}^{t-1})} 
        \subset \mathbb C^{m\cdot n} \times \operatorname{Grass}(n-t+1,n)$. 
Combining this with a map $F$ defining a determinatal
variety $X$ of type $(m,n,t)$ as in definition 2.3,
we obtain the following commutative diagram:
\begin{equation}
    \begin{xy}
        \xymatrix{
            Y=X\times_{M_{m,n}^t} W \ar[r]^{\;\;\;\;\;\;\hat F} \ar[d]_\pi &
            W \ar[d]_\rho \ar[dr]^{\hat P} &
            \\
            X \ar[r]^F &
            M_{m,n}^t \ar@{-->}[r]^{P}&
            {\mathbb P}^r\\
        }
    \end{xy}
    \label{eqn:DiagramTjurinaMod}
\end{equation}
If the dimension of $X$ is large enough to allow the exceptional locus 
to be a proper subset of $Y$, this is indeed a modification. Explicit equations
for $Y$ are given in \cite{FZ}, in the simplest case, $t=n \leq m$, the 
equations are given by
$$ F \cdot \underline{s} =0,$$
where $\underline{s} = (s_1,\dots,s_n)$ denotes the tuple of variables of 
$\operatorname{Grass}(n-1,n)=\operatorname{Grass}(1,n) = {\mathbb P}^{n-1}$.\\ 

Analogously the columns can be used for the same construction, as they
span a $(m-t+1)$-dimensional subspace in ${\mathbb C}^m$. 

\section{The discriminant for EIDS}\label{discr}

To study the discriminant, we need to detect the singularities of the
fibres. As already mentioned in the introduction, we shall decompose the
discriminant and compute the contributions separately. To this end, we shall 
exploit the smoothness criterion of Hironaka in a similar way as in \cite{BF}, 
but with a slightly more involved train of thought. For readers' convenience,
we postpone the general case and work out the key ideas in the smallest 
non-trivial case first: $2$-minors of $2 \times (2+k)$ matrices.

\begin{df} \cite{Hir} 
Let $(X,0) \subset ({\mathbb C}^n,0)$ be a germ with defining ideal 
$I_{X,0}$ generated by 
$f_1,\dots,f_s \in {\mathbb C}\{\underline{x}\}:=
  {\mathbb C}\{x_1,\ldots,x_n\}$, 
and assume that these power series form a standard basis of the 
ideal with respect to some local degree ordering. Assume further that 
the power series $f_i$ are sorted by increasing order. The tuple
$\nu^*(X,0) \in {\mathbb N}^s$ then denotes the sequence of orders of 
the $f_i$.
\end{df}

The tuple $\nu^*$ detects singularities, as the following lemma states, 
which is implicitly already present in Hironaka's work:

\begin{lem} The germ
$(X,0) \subset ({\mathbb C}^n,0)$ is singular at $p$ if and only if
$$\nu^*(X,0) >_{lex} (\underbrace{1,\dots,1}_{\operatorname{codim}(X)})$$ 
with respect to the lexicographical ordering $>_{lex}$.
\end{lem}

Of course, the above definition and lemma also make sense for the germ at 
any other point $p$ on $X$, as can be seen by moving the point $p$ to $0$ by 
a coordinate transformation and then passing to the germ. \\
 
Based on these considerations, we now want to decide whether for a given 
$2 \times (2+k)$-matrix $M$, defining a determinantal variety $X$ of type 
$(2,2+k,2)$,  $\nu^*(X,p) >_{lex} (1,\dots,1) $ at some point $p$. 
If $\nu^*(X,p)$ starts with  $1$ as first entry, the singularity needs to 
be essentially of type $(1,1+k,1)$ at $p$, as $M$ has to contain one 
entry which is of order zero and hence a unit in the 
local ring at $p$. In other words: if all entries of $M$ are of order at 
least $1$, the frist entry of $\nu^*(X,p)$ cannot be lower than $2$. 
We hence know that $X$ is singular at $p$, if and only if one of the 
following two alternatives holds: 
\begin{enumerate}
\item[(A)] the ideal $I_A$ generated by the entries of $M$ has order at 
least $1$ at $p$,
\item[(B)] $X$ is essentially of type $(1,1+k,1)$ and singular.
\end{enumerate}
In case (A), it suffices to determine where the ideal of $1$-minors of 
$M$ is of order at least $1$ to describe this contribution to the singular 
locus.\\ 
The second case is significantly more subtle: Such a unit 
might be sitting at any position in the matrix, which implies a priori that 
the respective contributions need to be computed for each matrix entry.
We know, however, that in case (B) the matrix is of rank precisely $1$ at $p$, 
i.e. all column vectors are collinear and hence correspond to a point in 
${\mathbb P}^1$. At such points, the Tjurina transformation is an 
isomorphism, which allows us to pass to the Tjurina transform, determine 
its singular locus outside $V(I_A)$ and take the closure thereof as the 
contribution (B).\\

Considering a larger matrix size, say a singularity of type $(m,n,m)$ with 
$n \geq m$, and the corresponding maximal minors, we can proceed analogously, 
but may a priori encounter $m$ cases corresponding to the singularity being 
essentially of type $(m-i,n-i,m-i)$ with $0 \leq i < m$. As larger minors 
also vanish, whenever all minors of a smaller size vanish, it suffices to
consider the vanishing locus of the $(m-1) \times (m-1)$ minors to determine
contribution (A). \\
For contribution (B), we need to assume that the rank of the matrix is 
precisely $m-1$. Hence its columns span a hyperplane in ${\mathbb C}^m$
and we can again make use of the fact that the Tjurina transform is an 
isomorphism at such points. So we can simply determine the singular locus
of the Tjurina transform outside $V(I_A)$ and take the closure thereof as we 
did for contribution (B) in the previous case. To give a concise overview of 
the necessary compuations, this is summarized in algorithm 
\ref{alg_discr_max}. There the input is restricted to polynomials for purely 
practical reasons: it should consist of finitely many terms.

\begin{algorithm}[h]
\caption{Discriminant for EIDS of type $(m,n,m)$ (sequential)}
\label{alg_discr_max}
\begin{algorithmic}[1]

\REQUIRE $M \subset \operatorname{Mat}(m,n;{\mathbb C}[x_1,\dots,x_r])$, $m\leq n$ 
         defining EIDS at $\underline{0}$
\ENSURE  ideals $I_A$, $I_B$ describing the discriminant of the versal family 
         of the given EIDS as follows:\\
         \begin{itemize}\itemsep0pt
         \item $I_A$ describes contribution (A) 
         \item $I_B$ describes contribution (B)
         \item $I_A \cap I_B$ describes the discriminant  
         \end{itemize}

\smallskip

\STATE matrix $N := versalG(M)$
\STATE ideal  $I_A := \langle$ (m-1)-minors of $N \rangle$
\STATE $I_A = eliminate(I_A;x_1,\dots,x_r)$ 
\STATE ideal  $I_{Tj} := (s_1,\dots,s_m) \cdot N$, with generators denoted 
              as $f_1, \dots,f_n$
\STATE ideal  $I_B := I_{Tj} + minor\left(\left(\frac{\partial f_i}{\partial v_j}\right)_{i,j},n-m+1\right)$ where $\underline{v}=(\underline{x},\underline{s})$
\STATE $I_B = (I_B:\langle \underline{s} \rangle^\infty)$
\STATE $I_B = eliminate(I_B;x_1,\dots,x_r,s_1,\dots,s_n)$
\STATE $I_B = (I_B : I_A^\infty)$
\RETURN ($I_A$,$I_B$)

\end{algorithmic}
\end{algorithm}

The algorithm \ref{alg_discr_max} requires a saturation in step 6 to
remove any contribution of the irrelevant ideal. As this can be a significant 
bottelneck, a parallel approach can be helpful: replace step 6 by running
step 7 in all charts $D(s_i)$ of the projective space and intersect the 
resulting ideals $I_{B,i}$ to obtain $I_B$. 

\begin{rem}
In lines 3 and 7 of \ref{alg_discr_max}, we use elimination which means that 
we endow the resulting complex space with the annihilator structure (cf. 
\cite{GLS}, Def. 1.45), which is not compatible with base change. We might as 
well have chosen the Fitting structure relying on resultant methods instead of 
elimination, as this is compatible with base change.\\
The choice of elimination over resultants is mostly based on the purely 
practical fact that the implementation of elimination in {\sc Singular} is
significantly more refined than the one of resultants.
\end{rem}

\begin{example}
\begin{itemize}
\item[(a)] 
Consider the determinantal singularity defined by the $2$-minors of a matrix 
of the form
$$M=\begin{pmatrix}
    x_1 & \dots & x_{r-1} & x_r \cr
    x_{r+1} & \dots & x_{2r-1} & f(x_1,\dots,x_{r-1})
    \end{pmatrix}.$$
The ${\mathcal G}$-versal family with this special fiber can be written as
$$M_{\underline{t}}=\begin{pmatrix}
    x_1+a_1 & \dots & x_{r-1}+a_{r-1} & x_r \cr
    x_{r+1} & \dots & x_{2r-2} & F(x_1,\dots,x_{r-1},a_r,\dots,a_s)+a_0
    \end{pmatrix},$$
where $F(x_1,\dots,x_{r-1},a_r,\dots,a_s)$ corresponds to a versal deformation 
with section of $f(\underline{x})$. For determining the contributions to 
the discriminant, we now apply our algorithm and obtain:
\begin{eqnarray*}
I_A& = & \langle x_1+a_1, \dots, x_{r-1}+a_{r-1},
                 x_r, \dots, x_{2r-1}, 
                 F(a_1,\dots,a_s)+a_0 \rangle 
                 \cap {\mathbb C}[\underline[{a}] \\
   & = & \langle F(\underline{a})+a_0 \rangle\cr
I_B& = & {\rm discriminant\;\; of\;\;  F(x_1,\dots,x_{r-1},a_r,\dots,a_s)+a_0}
\end{eqnarray*}
The locus, above which we see determinantal singularities, is the smooth 
hypersurface $V(F(a_1,\dots,a_s)+a_0)$ and the remaining part of the 
discriminant is precisely the discriminant of the versal family with section
in the right hand lower entry.
\item[(b)] As the next example, we consider 3 families of ICMC2 singularities:
$$M_1=\begin{pmatrix}
      x & y & z \cr
      w & x & y+f(v)
      \end{pmatrix} \;\;\;
  M_2=\begin{pmatrix}
      x & y & z \cr
      w & u & x+f(v)
      \end{pmatrix}$$
  $$M_3=\begin{pmatrix}
      x & y & z \cr
      w & x+g(u,v) & y+h(u,v)
      \end{pmatrix},$$
where $f(v)=v^k$ for some $k \in {\mathbb N}$ and $(g(u,v),h(u,v))$ describes
a fat point in the plane. Then, $M_1$ is a $3$-fold in $({\mathbb C}^5,0)$ and 
the other two are $4$-folds in $({\mathbb C}^6,0)$. Direct computation yields
the following versal families:
\begin{eqnarray*}
M_{1;\underline{a}, \underline{b}}& = &\begin{pmatrix}
      x & y & z \cr
      w & x+\sum_{i=0}^{k-1} a_iv^i & y+v^k+\sum_{i=0}^{k-2} b_iv^i
      \end{pmatrix}\\
M_{2;\underline{b}} & = & \begin{pmatrix}
      x & y & z \cr
      w & u & x+v^k+\sum_{i=0}^{k-2} b_iv^i
      \end{pmatrix} \\
M_{3;\underline{t}} & = & \begin{pmatrix}
      x & y & z \cr
      w & x+G(u,v,\underline{t}) & y+H(u,v,\underline{t})
      \end{pmatrix},
\end{eqnarray*}
with suitably chosen $G(u,v,\underline{t})$ and $H(u,v,\underline{t})$ 
(cf. \cite{FN}). The Tjurina transforms for the three cases are:
\begin{eqnarray*}
I_{Tj,1} & = & \langle  sx+tw,sy+t(x+\sum_{i=0}^{k-1} a_iv^i),
                        sz+t(y+v^k+\sum_{i=0}^{k-2} b_iv^i) \rangle \cr
I_{Tj,2} & = & \langle  sx+tw,sy+tu,
                        sz+t(x+v^k+\sum_{i=0}^{k-2} b_iv^i) \rangle \cr
I_{Tj,3} & = & \langle  sx+tw,sy+t(x+G(u,v)),
                        sz+t(y+H(u,v))
\end{eqnarray*} 
Passing to the two affine charts of ${\mathbb P}^1$, we immediately see that
all the $V(I_{Tj,i})$ are non-singular. Hence, $I_B=\langle 1 \rangle$ and
$I_A$ describes the whole discriminant in these cases.
\item[(c)] To illustrate contributions (A) and (B) not only in the extremal 
cases shown above, we give two surface and two 3-fold examples from the list 
of simple ICMC2 singularities \cite{FN}, which for the surface case conincides
with Tjurina's list of rational triple point singularities in \cite{Tju}.\\
As first example of a surface singularity, we consider Tjurina's $A_{0,1,2}$ 
singularity. Its versal family is:
$$\begin{pmatrix}
  x_3 & x_4+a_4 & x_2^2+a_5\cr
  x_4^3+x_4a_1+a_2 & x_2+a_3 & x\end{pmatrix}$$
for which the two contributions to the discriminant are:
\begin{eqnarray*}
I_A & = & \langle a_3^2+a_5, a_4^3+a_1a_4-a_2 \rangle\\
I_B & = & \langle a_5(4a_1^3+27a_2^2)\rangle .
\end{eqnarray*}
So contribution (A) is a $5$-dimensional smooth subvariety of the base and 
contribution (B) consists of two hypersurfaces, a smooth one and a cylinder 
over a plane cusp. \\
As next example, we consider Tjurina's $D_0$ singularity:
$$\begin{pmatrix}
  x_3+a_1 & x_2+a_5 & x_1 \cr
  x_1+x_3a_3+x_4a_2+a_4 & x_4+a_6 & x_3^2+x_2x_4+a_7
  \end{pmatrix}, $$
for which a direct computation yields 
$$I_A = \langle a_1a_3+a_2a_6-a_4,a_1^2+a_5a_6+a_7 \rangle,$$
which again happens to be a smooth subvariety of codimension $2$. Contribution (B) is an irreducible hypersurface of degree $16$, of which we do not give the
explicit equations here.\\
The first 3-fold example has the versal family
$$\begin{pmatrix}
  x_5^3+x_4x_5+x_4a_1+a_2 & x_2+x_5a_3+a_4 & x_1 \\
  x_3 & x_4+x_5a_5+a_6 & x_5^3+x_2x_5+a_7
  \end{pmatrix}. $$
Here contribution (A) is an irreducible hypersurface of degree 7, whereas
contribution (B) is the hypersurface defined by
$$I_B = \langle a_7(a_1^3-a_2) \rangle.$$
The versal family in the final example is:
$$\begin{pmatrix}
  x_3+a_1 & x_5^2+x_1 & x_2+x_5a_4+a_5 \\
  x_1+x_5a_2+a_3 & x_3^2+x_2x_5+x_4a_6+a_7 & x_4
  \end{pmatrix}.$$
In this case, contribution (A) is an irreducible hypersurface of degree 5 and
contribution (B) is an irreducible hypersurface of degree 14.
\end{itemize}
\end{example}

Up to now, we had restricted our considerations to ideals of maximal minors to
allow a clearer exposition of the material. For considering non-maximal minors,
we first observe that the ${\mathbb P}^{m-1}$ which was used in case (B) above 
is just a manifestation of a Grassmannian in the simplest case, hyperplanes in 
${\mathbb C}^m$. Passing to non-maximal minors, however, the Grassmannian has 
more structure which we need to recall before continuing with our study of the 
discriminant.\\

Classically the Grassmannian describing the set of $r$-dimensional linear 
subspaces of an $n$-dimensional vector space $V$ or equivalently of 
$(r-1)$-dimensional linear subspaces 
${\mathbb P}^{r-1} \subset {\mathbb P}^{n-1}$ can be embedded into projective 
space by the Pl\"ucker embedding:
\begin{eqnarray*}
\operatorname{Grass}(r,n) & \longrightarrow & {\mathbb P}\left(\bigwedge^r V \right) \cong
              {\mathbb P}^{{n \choose r} -1}\\
\operatorname{span}(v_1,\dots,v_r) & \longmapsto & v_1 \wedge \dots \wedge v_r
\end{eqnarray*}
The image of this embedding is closed; the equations of the image 
are quadratic in the variables of ${\mathbb P}^{{n \choose r} -1}$: 
we denote the variables as $x_{i_1,\dots,i_r}$ for any given sequence of 
indices $1 \leq i_1 < i_2 < \dots < i_r \leq n$. Purely for convenience of 
notation, we extend this to any subset of $\{1,\dots,n\}$ with $r$ elements. 
To this end, we set $x_{i_1,\dots,i_r}=0$, if two elements of the 
index conincide and postulate that permutations of indices change the
sign by the sign of the permutation. Then each Pl\"ucker relation is of the
form
$$\sum_{j=0}^r (-1)^j x_{i_1,\dots,i_{r-1},k_j} \cdot 
                    x_{k_0,\dots,{\hat k_j},\dots,k_r} =0$$
where $i_1,\dots,i_{r-1}$ and $k_0,\dots,k_r$ are subsets of $\{1,\dots,n\}$,
i.e. the ideal of the image of the Pl\"ucker embedding is generated by quadratic 
polynomials.
At this point it is important to stress that a Pl\"ucker coordinate 
$x_{i_1,\dots,i_r}$ can be interpreted as the $r$-minor of the matrix with
columns $v_1,\dots,v_r$ involving the rows $i_1,\dots,i_r$. \\

Now we are ready to consider the general case of an EIDS of any type 
$(m,n,t)$ with  $1 < t  \leq m \leq n$. The case (A) does not present 
additional difficulties here, as the ideal $I_A$ can be computed directly 
as $(t-1)$-minors of the given matrix $M$. As we have seen before, case (B) 
comprises all singular points at which the matrix $M$ has rank precisely $t-1$.
Appending $t-1$ columns, whose entries are $n(t-1)$ new variables, and imposing
the condition that at least one $(t-1)$-minor of this part does not vanish, 
allows us to restrict to this part of $X$ by taking the $t$ minors of the new
matrix. With our previous considerations about the Grassmannian, this can 
also be more conveniently expressed by introducing Pl\"ucker coordinates 
instead of the additional columns and then leads to a set of equations of 
the form
$$\sum_{l=1}^t (-1)^l y_{i_1,\dots,{\hat i_l},\dots,i_t} m_{i_l,j} = 0$$
for all strictly increasing $t$-tuples $\{i_1,\dots,i_t\} \subset 
\{1,\dots,m\}$ and for all $j \in \{1,\dots,n\}$, where $m_{l,j}$ denotes 
the entry of $M$ at the position $(l,j)$. We now consider the ideal generated
by these polynomials and by the generators of the image of the Pl\"ucker embedding. 
After saturating out the irrelevant ideal, the new ideal describes the part of 
$X$ which is relevant 
for contribution (B). We can then compute the singular locus thereof, saturate 
out the maximal ideal in the $y_{i_1,\dots,i_t}$ and then eliminate the orignal 
variables $\underline{x}$ and all variables $y_{i_1,\dots,i_t}$ as before to 
obtain $I_B$. For practical 
purposes, a parallel approach using a covering of the Grassmannian with affine charts 
should again be the choice in implementations due to the extremely high 
number of variables and the particularly simple structure of the ideal of 
the Grassmannian in each chart.\\

Comparing the construction above with the general construction of the Tjurina
transform in \cite{FZ}, we see that the use of the Grassmannian in both 
settings is the same and that $I_B$ captures precisely the singular locus of 
the Tjurina transform as before. Therefore, we have decomposed the 
discriminant of a determinantal singularity in the following way:

\begin{prop}
Let $(X,0)\subset {\mathbb C}^N$ be a determinantal singularity of type 
$(m,n,t)$, $m>n$, defined by $F^{-1}(M_{m,n}^t)$ and let ${\mathcal X}$ be its 
${\mathcal G}$-versal family. Further assume that $dim(X) \geq m$. 
Then the discriminant of ${\mathcal X}$ decomposes naturally into two 
contributions:
\begin{enumerate}
\item[(A)] points in the base space, above which there are determinantal 
      singularities
\item[(B)] points in the base space, above which there are singularities 
           leading to singular points in the Tjurina transform.
\end{enumerate}
\end{prop}

The condition on the dimension of $X$ in the preceding proposition ensures that 
the exceptional locus of the Tjurina modification is a lowerdimensional 
closed subset of the Tjurina transform. In the above decomposition, the
contribution related to the Tjurina transform may be empty in some cases,
whereas the other one always contains at least the origin. 

\section{Applications to the ICMC2 case}

The above considerations not only yield a decomposition of the discriminant. 
They show that determinantal singularities $(F^{-1}(M_{m,n}^t),0)$ possess in 
general two kinds of contributions to the singular locus: The structural
contribution arising from $F^{-1}(M_{m,n}^{t-1})$, to which the Tjurina 
transform is partly blind, and a contribution arising from the map $F$ itself,
which manifests itself in the Tjurina transform. The well-studied special
case of ICMC2 singularities, in which ${\mathcal G}$-versality and versality
are known to conincide, provides a good setting to consider this in more 
detail and illustrate the consequences.  

\begin{lem}
Let $(X,0) \subset ({\mathbb C}^N,0)$ be an ICMC2 singularity, i.e. of type 
$(t,t+1,t)$, with generic linear entries.
\begin{enumerate}
\item It has a smooth Tjurina transform, if and only if $N \geq 2t$.
\item The singular locus of the Tjurina transform is a determinantal variety
      of type $(t+1,N,t+1)$ of dimension $2t-N-1$ in ${\mathbb P}^{t-1}$ for
      $t+1 \leq N < 2t$. 
\end{enumerate}
\end{lem}

\begin{proof}
In the case of generic linear entries in the matrix $M$ of $X$, the ideal of 
the Tjurina transform $Y$ is generated by bi-homogeneous polynomials of bidegree, 
$(1,1)$. Therefore the Jacobian matrix of it is of the form
$$\begin{pmatrix}
   A & \mid & M^T
   \end{pmatrix}, $$
where the first columns hold the derivatives w.r.t. the original variables
and the remaining ones the derivatives w.r.t. the variables of the 
${\mathbb P}^{t-1}$. Then $A$ is a $(t+1) \times N$ matrix with homogeneous
entries of degree $1$, which only involve the variables of the 
${\mathbb P}^{t-1}$ and which are generic, because the entries of $M$ were 
generic. As the singular locus of $X$ is just the origin, the Tjurina 
modification is an isomorphism outside the origin and we therefore only
need to evaluate the Jacobian criterion above $V(\underline{x})$. This causes 
the last $t$ columns of the Jacobian matrix and all generators of the ideal 
of the Tjurina transform to vanish. Hence the singular locus of the Tjurina
transform is precisely the vanishing locus of the maximal minors of $A$. \\
For the first claim, it suffices to observe that the codimension of the 
singular locus of $Y$ is $N-t$ in $\underline{0} \times 
{\mathbb P}^{t-1}$, which needs to exceed $t-1$ for $Y$ to be smooth, 
i.e. we obtain the condition $N>2t-1$. These arguments also prove the second claim.
\end{proof}

As the matrix describing the singular locus of the Tjurina transform is a 
square matrix for $N=t+1$, we immediately get the following corollary:

\begin{corollary}
The singular locus of the Tjurina transform of an ICMC2 singularity 
$(X,0) \subset ({\mathbb C}^{t+1},0)$ 
of type $(t,t+1,t)$ with generic linear entries is a hypersurface of degree 
$t+1$ in ${\mathbb P}^{t-1}$.
\end{corollary}

For the other extreme of $N=2t-1$, i.e. for isolated singular points in the 
Tjurina transform, it is also possible to determine the number of points 
as it coincides with the value of the only non-zero term in the Hilbert 
Polynomial, the constant term, in this case. This polynomial itself can be
obtained from a graded free resolution (given by the Eagon-Northcott complex
for determinantal varieties of type $(m,n,m)$). We only give an example of such
a computation:

\begin{corollary}
The singular locus of the Tjurina transform of an ICMC2 singularity 
$(X,0) \subset ({\mathbb C}^{5},0)$
of type $(3,4,3)$ with generic linear entries consists of 10 points in general
position in ${\mathbb P}^2$.
\end{corollary}

\begin{proof}
It is well-known that the Hilbert polynomial can be read off from the Betti 
diagram of a minimal free resolution (see e.g. \cite{GPf} or \cite{EisSyz}).
Here the situation is particularly simple: the choice of $N=5$ and $t=3$  
leads to a singular locus $\Sigma$ of the Tjurina transform which only consists
of points in ${\mathbb P}^2$ and can be described by the vanishing of the 
$4$-minors of a $5 \times 4$ matrix with generic linear entries. 
In particular, this is the Hilbert-Burch case for which von Bothmer, Bus\'e 
and Fu give an even more explicit formula in \cite{BBF}: 
Given the minimal graded free resolution
$$ 0 \longrightarrow \bigoplus_{i=1}^t {\mathcal O}_{{\mathbb P}^2}(-l_i)
     \longrightarrow \bigoplus_{i=1}^{t+1}{\mathcal O}_{{\mathbb P}^2}(-k_i)
     \longrightarrow I_{\Sigma} \longrightarrow 0,  $$
the Hilbert polynomial and hence the number of points is
$$\frac{\sum_{i=1}^{t} l_i^2 - \sum_{i=1}^{t+1} k_i^2}{2}.$$
In our setting, all $l_i$ are have the value $-5$ and all $k_i$ are $-4$ which
yields $\frac{1}{2} (4\cdot 25 - 5 \cdot 16) = 10$ points.
\end{proof} 

As the Tjurina transform can be non-singular, there are cases in which the
contribution $(B)$ of the discriminant of the versal family is empty as e.g. 
for the generic ICMC2 of type $(2,3,2)$ in $({\mathbb C}^k,0)$ for $k \geq 4$. 
But there are, of course, many non-generic matrices with singular Tjurina 
transform even in these dimensions. The smoothness of the Tjurina transform 
is actually a statement about the adjacencies of an EIDS, as the following 
lemmata show:

\begin{lem}
Let $(X,0) \subset ({\mathbb C}^k,0)$, be an EIDS for which contribution (B) 
to the discriminant of the versal family is not empty, then $(X,0)$ is 
adjacent to an $A_1$ singularity and has a smoothing passing through $A_1$
singularities.
\end{lem}

\begin{proof}
If $p$ is a point in the base of the versal deformation belonging to
contribution (B), then the Tjurina transform of the fibre above this point is 
singular and has only ICIS singularities, which are themselves adjacent to an $A_1$.
Moreover, as contribution (B) is non-empty, it contains by construction an open set 
which does not meet contribution (A). Above this open set, there are no fibres with
determinantal singularities, and the Tjurina modification is already an isomorphism
for these fibres. Hence the original singularity is also adjacent to an $A_1$ and 
possesses a smoothing which passes through $A_1$ singularities. 
\end{proof}

\begin{rem}
There are smoothable EIDS for which no smoothing passes through an $A_1$ 
singularity as can be seen from the results in \cite{FZ} and \cite{Zac}: For surface 
singularities of type $(2,3,2)$ in $({\mathbb C}^4,0)$ this is precisely the 
determinantal singularity with generic linear entries. 
For 3-fold singularities of type $(2,3,2)$ in $({\mathbb C}^5,0)$ these are
precisely the singularities with $b_3-b_2=-1$. As we always 
have $b_2=1$, this difference implies $b_3=0$, whence the Tjurina transform 
is smooth, no adjacency to an $A_1$ is possible and the contribution (B) to
the discriminant is empty. However, these singularities are smoothable
through a different mechanism: They pass through the EIDS with generic 
linear entries of the appropriate ambient dimesion. For the latter, any
non-trvial deformation is a smoothing.\\
In dimensions, in which the determinantal singularity with generic linear
entries is a rigid EIDS, the contribution (B) will always
be empty, as the terminal object in the adjacency diagramm is a rigid 
determinantal singularity which causes contribution (A) to be the
whole base of the versal family. Therefore, we cannot decide in general 
whether a singularity is adjacent to an $A_1$ based solely on the fact that
contribution (B) is empty. But even in this case, it can make sense to 
consider the locus above which there are ICIS singularities, by omitting the
final saturation by $I_A$ in the computation of contribution (B).
\end{rem}

These last observations also indicate that passing to the Tjurina transform 
provides valuable information about the original singularity, but this 
information
also relies on knowledge about the contribution (A). The cases of surfaces in 
$({\mathbb C}^4,0)$ and $3$-folds in $({\mathbb C}^5,0)$ show how different the
behaviour can be. To illustrate this, we first discuss Wahl's conjecture about
the relation between Milnor and Tjurina number in the surface case: we reprove
the easier direction that quasihomogeneity implies $\mu = \tau -1$ for the 
special situation that there are only isolated singularities in the Tjurina 
transform. This already indicates, where it might be possible to find 
counterexamples for the other direction, i.e. non-quasihomogeneous codimension $2$ surface singularities satisfying $\mu =\tau -1$, and we pursue this thought
to construct a whole class of counterexamples. 
Contrasting the rather controlled situation of surfaces, 
we then give an explanation of the observations of Damon and Pike \cite{DP} in the 
$3$-fold case, relying on the same mechanism, but with very different outcome. 

\begin{lem} \label{Wahl_ok}
Let $(X,0) \subset ({\mathbb C}^4,0)$ be a quasihomogeneous isolated determinantal singularitiy 
of type $(2,3,2)$ with at most isolated singularities in the Tjurina transform. Then 
$$\mu = \tau - 1.$$
\end{lem}

\begin{proof}
In this proof we denote the Tjurina transform of $X$ by $Y$ and we assume that 
the presentation matrix of $X$ is chosen to have quasihomogeneous entries and 
respect row and column weights as in \cite{FK1}. This implies that $Y$ is quasihomogeneous w.r.t. the same weights.
 
From \cite{FZ}, we know that 
$$\mu(X) = 1 + \sum_{p \in \operatorname{Sing}(Y)} \mu(Y,p),$$ 
i.e. it differs from the sum over the Milnor numbers of the singularities 
$(Y,p)$ of the Tjurina transform by 1. 
The singularities of the Tjurina transform are at most ICIS singularities 
and hence satisfy 
$\mu(Y,p) \geq \tau(Y,p)$ with equality precisely in the case
of quasihomogeneous singularities \cite{LoSt}. So it remains to establish 
the relation 
$$\tau(X) = \sum_{p \in \operatorname{Sing}(Y)} \tau(Y,p) + 2$$
to prove the claim. However, after a few preliminary considerations this leads to a 
Gr\"obner basis computation which we will sketch for a general matrix of the given 
properties in the rest of the proof.\\

To this end, we first recall from \cite{FZ} that
$$  T^1(X) \cong N' = H^1(Y_0,T_{Y_0}) \oplus 
            \bigoplus_{p \in \operatorname{Sing}(Y)} T^1(Y,p)$$
implying for the corresponding dimensions
$$\tau(X) = \operatorname{dim}_{\mathbb C} H^1(Y, TY) 
           +\sum_{p \in \operatorname{Sing}(Y)} \tau(Y,p).$$ 
As all of these 
${\mathbb C}\{\underline{x}\}$-modules are finite dimensional 
${\mathbb C}$-vector spaces, this also induces an isomorphism of  
${\mathbb C}$-vector spaces which can be expressed in terms of a monomial 
basis of $T^1_X$. To complete the proof, we therefore need to identify those
basis elements in $T^1_X$ which do not contribute to $\oplus_{p \in \operatorname{Sing}(Y)} T^1_{Y,p}$.\\

To keep the presentation of the rest of the proof as simple as possible, we denote
the variables by $x,y,z,w$ and denote the tuple of these four variables by 
$\underline{x}$ in the following.
Since the Tjurina transform only contains isolated singularities, $X$ can be
expressed in terms of a matrix 
$$A = \begin{pmatrix} x  &  y  &  z \cr
                      a  &  b  &  c \end{pmatrix}$$
according to \cite{FZ}, where 
$a,b,c \in {\mathfrak m} \subset {\mathbb C}\{\underline{x}\}$ and no term of
$a$ is divisible by $x$. 
\\
The Tjurina transform is described by the ideal 
$I_{Tj}= \langle sx+ta,sy+tb,sz+tc \rangle$
with Jacobian matrix
$$\begin{pmatrix}
   s    & ta_y   & ta_z   & ta_w & x & a\cr
   tb_x & s+tb_y & tb_z   & tb_w & y & b\cr
   tc_x & tc_y   & s+tc_z & tc_w & z & c
  \end{pmatrix}, $$
where a subscript stands for the partial derivative by the respective variable. 
At $t=0$ there is a $3$-minor $s^3$, whence $(0,0,0,0) \times (1:0)$ cannot
be a singular point of the Tjurina transform. Therefore it suffices to 
consider the chart $D(t)$. We know from \cite{FZ} that 
$$N'= (({\mathbb C}[s,t]\{\underline{x}\})^3/K)_{(1)}$$
where the subscript $(1)$ denotes the degree $1$ part in $s$ and $t$ and the module $K$
is generated by the columns of the Jacobian matrix of above and  
$I_{Tj} \cdot {\mathbb C}\{\underline{x}\}^3$. As we are interested in $K_{(1)}$ and
possibly slices of higher degree, but not in $K_{(0)}$,
we now replace the columns $5$ and $6$ of the above matrix by their multiples with
$s$ and $t$.\
To fix a numbering of the generators of $K_{(1)}$, we keep the resulting numbering
of the columns: starting with the partial derivatives by the $x,y,z$ and $w$ and 
continuíng with the $s$ and $t$ multiples we just introduced, the first eight 
generators are the columns of the following matrix:
$$\begin{pmatrix}
   s    & ta_y   & ta_z   & ta_w & sx & sa & tx & ta\cr
   tb_x & s+tb_y & tb_z   & tb_w & sy & sb & ty & tb\cr
   tc_x & tc_y   & s+tc_z & tc_w & sz & sc & tz & tc
  \end{pmatrix}. $$
The last $9$ generators are then ordered as in the following matrix:
$$\begin{pmatrix}
sx+ta &sy+tb & sz+tc &   0   &      & 0 \cr
  0   &  0   &   0   & sx+ta &\dots & 0 \cr
  0   &  0   &   0   &   0   &      & sz+tc
\end{pmatrix}.$$ 
We also know that the sum over the $T^1(Y,p)$, which are all sitting above the origin 
of ${\mathbb C}^5$, can be computed as the ($\underline{x}$-local, but $s$-global) Tjurina 
module in the respective chart of ${\mathbb P}^1$ by considering the module 
$$T=({\mathbb C}[s]\{\underline{x}\})^3/\overline{K}$$
obtained by dehomogenizing the previous module w.r.t. the variable $t$. Our task will
now be a comparison of  $K_{(1)}$ and $\overline{K}$ and of the respective quotients
by means of the corresponding leading ideals, which we obtain from a standard basis
computation. \\

By the assumption that the entries of the original matrix are contained 
in ${\mathfrak m}$, only the first $4$ generators of $K$ and
$\overline{K}$ can possibly contain $\underline{x}$-degree zero entries. Using a mixed 
ordering which first compares w.r.t. a global ordering $s >_{lex} t$, then a negative 
lexicographical ordering in $x < y < z < w$ and finally a module ordering, it is 
now easy to see that the first three generators have $s$ as entry in their leading 
monomial and that the leading monomials  are in pairwise different entries not 
allowing any non-vanishing s-polynomial among these. We can thus directly
eliminate all $s$-terms from the other generators by reducing with these three 
generators and assume from now on that generator $4$ and all further ones do not 
involve any $s$ -- both in  $K$ and in $\overline{K}$. From now on, we use these 
reduced columns, not involving $s$, instead of the original columns $4$ to $17$ 
denoting the $i$-th columns thereof by $C_i$. 

We immediately see that any 
s-polynomial computation arising from $C_i$ and $C_j$ with $i,j \geq 4$ is identical 
for $K_{(1)}$ and $t \cdot \overline{K}$; moreover, it does not involve any $s$ as
we had chosen an elimination ordering for $s$. Therefore the standard bases for 
$\langle C_4,\dots,C_{17} \rangle$ and its dehomogenization w.r.t. $t$ are in 1:1 
correspondence. 
Considering an s-polynomial between $C_i$ and $C_j$ with 
$5 \leq i \leq 17$ and the respective $1 \leq j \leq 3$, a direct computation shows 
that its normal form w.r.t. $\{C_1,C_2,C_3\}$ already lies in 
the module  $t \cdot \langle C_5,\dots , C_{17} \rangle_{{\mathbb C}\{\underline{x}\}}$
and hence reduces to zero w.r.t. a standard basis generated by $C_4,\dots,C_{17}$. 
The respective relations are stated in the table below. There the abbreviation $\operatorname{jacob}$ 
denotes the jacobian matrix and $j\in\{1,2,3\}$ stands for the component in which the leading term 
is found, i.e. the index of the appropriate $C_j$ for the s-polynomial.\\[0,2cm]
\begin{tabular}{lcl} \label{tab1}
NF(spoly$(C_5,C_j),\{C_1,C_2,C_3\})$ &=& $t(C_6 + a_xC_9 + a_yC_{10} + a_z C_{11} + b_xC_{12}$ \cr
    & + & $b_y C_{13} + b_zC_{14} + c_xC_{15}+c_yC_{16}+c_zC_{17})$\cr
NF(spoly$(C_6,C_j),\{C_1,C_2,C_3\})$ &=& $t((a_x+b_y+c_z)C_6 + \operatorname{det}(\operatorname{jacob}(a,b,c))C_7 $\cr
    & + & $ (a_yb_x-a_xb_y+a_zc_x+b_zc_y-a_xc_z-b_yc_z)C_8 $\cr
    & + & $ (b_zc_y-b_yc_z)C_9 + (a_yc_z-a_zc_y)C_{10} $\cr
    & + & $ (a_zb_y-a_yb_z)C_{11} + (b_xc_z-b_zc_x)C_{12} $\cr
    & + & $ (a_zc_x-a_xc_z)C_{13} + (a_xb_z-a_zb_x)C_{14} $ \cr
    & + & $ (b_yc_x-b_xc_y)C_{15} + (a_xc_y-a_yc_x)C_{16} $\cr
    & + & $(a_yb_x-a_xb_y) C_{17})$\cr
NF(spoly$(C_7,C_j),\{C_1,C_2,C_3\})$ &=& $t(C_8+C_9+C_{13}+C_{17})$\cr
NF(spoly$(C_8,C_j),\{C_1,C_2,C_3\})$ &=& $t C_6$\cr
NF(spoly$(C_9,C_1),\{C_1,C_2,C_3\})$ &=& $t(a_xC_9 + b_xC_{12} +c_xC_{15})$\cr
NF(spoly$(C_{10},C_1),\{C_1,C_2,C_3\})$&=& $t(a_xC_{10} + b_xC_{13} +c_xC_{16})$\cr
NF(spoly$(C_{11},C_1),\{C_1,C_2,C_3\})$&=& $t((a_xC_{11} + b_xC_{14} +c_xC_{17})$\cr
NF(spoly$(C_{12},C_2),\{C_1,C_2,C_3\})$&=& $t(a_xC_9 + b_xC_{12} +c_xC_{15})$\cr
NF(spoly$(C_{13},C_2),\{C_1,C_2,C_3\})$&=& $t(a_xC_{10} + b_xC_{13} +c_xC_{16})$\cr
NF(spoly$(C_{14},C_2),\{C_1,C_2,C_3\})$&=& $t((a_xC_{11} + b_xC_{14} +c_xC_{17})$\cr
NF(spoly$(C_{15},C_3),\{C_1,C_2,C_3\})$&=& $t(a_xC_9 + b_xC_{12} +c_xC_{15})$\cr
NF(spoly$(C_{16},C_3),\{C_1,C_2,C_3\})$&=& $t(a_xC_{10} + b_xC_{13} +c_xC_{16})$ \cr
NF(spoly$(C_{17},C_3),\{C_1,C_2,C_3\})$&=& $t((a_xC_{11} + b_xC_{14} +c_xC_{17})$ \cr
\end{tabular}\\[0,2cm]
This only leaves potential differences between standard bases of
$ K_{(1)} $ and $\overline{K}$ in s-polynomials arising from $C_4$ and $C_i$ for 
$ 1 \leq i \leq 3$.\\

In $K_{(1)}$, we can see no contributing s-polynomial which arises between $C_i$ and $C_4$ 
with $1 \leq i \leq 3$, because its leading monomial would be in 
$(s,t)$-degree $2$. In $\overline{K}$ on the other hand, such an s-polynomial is relevant, 
as $C_4$ is non-zero (due to the need for a pure power in $w$ to appear to allow finite dimension), 
is of $s$-degree zero and has lower $w$-order in each entry than the corresponding entry of 
the second row of $M$. 
Considering this more closely, we see that $sC_4$ and $s^2C_4$ can both contribute 
to relevant s-polynomials with $C_1, C_2, C_3 $.
On the other hand, using the linear combinations of $C_1,C_2,C_3$ indicated by the 
columns of the right adjoint of the $3 \times 3$ square matrix with columns $C_1,C_2,C_3$,
the minor of the Jacobian matrix corresponding to $C_1,C_2,C_3$ reduces to zero in all
entries. Hence $s^3C_4$ does not provide any new contribution but reduces to lower degree 
terms in $s$. Therefore we obtain precisely two $(s,t)$-degrees for which the standard 
basis of $\overline{K}$ contains elements not necessarily appearing in the one of 
$K_{(1)}$. This implies that their leading monomials can be part of the computed 
monomial basis of $N'$, but reduce to zero in $T=({\mathbb C}[x]\{\underline{x})^3/\overline{K}$. 
If they are in the monomial basis of $N'$, they contribute to ${\operatorname{dim}}_{\mathbb C}H^1(Y, T_Y)$.\\

Showing that no $\underline{x}$-multiple of two leading monomials of the still remaining s-polynomials
is non-zero in $N'$, we obtain  
$\operatorname{dim}_{\mathbb C}H^1(Y, T_Y) \leq 2$ which in turn proves one side of the inequality 
in the original claim. Again, we simply state the relations:
\begin{eqnarray*}
x \cdot {\rm NF}({\rm spoly}(sC_4,\{C_1,C_2,C_3\})) & = & a_wtC_9+b_wtC_{12}+c_wtC_{15}-atC_4\cr
y \cdot {\rm NF}({\rm spoly}(sC_4,\{C_1,C_2,C_3\})) & = & a_wtC_{10}+b_wtC_{13}+c_wtC_{16}-btC_4\cr
z \cdot {\rm NF}({\rm spoly}(sC_4,\{C_1,C_2,C_3\})) & = & a_wtC_{11}+b_wtC_{14}+c_wtC_{17}-ctC_4
\end{eqnarray*}
Of course, such relations continue to hold after multiplication with $s$ and subsequent
reduction by $C_1, C_2, C_3$, which now leaves only one case to be considered:
$$w \cdot {\rm NF}({\rm spoly}(sC_4,\{C_1,C_2,C_3\})),$$
but by the Euler relation this expression reduces to zero, as we are considering the 
quasihomogeneous case.

For proving the equality part of the statement, it suffices to establish
$$\tau(X) = \sum_{p \in \operatorname{Sing}(Y)} \tau(Y,p) + 2.$$
Thus we need to show that $\operatorname{dim}_{\mathbb C} H^1(Y,T_Y)$ is at 
least $2$ for all quasihomogeneous ICMC2 surface singularities of type $(2,3,2)$ which 
have at most isolated singularities in their Tjurina transform. This 
certainly holds for the simplest such singularity given by the matrix 
$$\begin{pmatrix} x_1 & x_2 & x_3 \cr x_4 & x_1 & x_2
\end{pmatrix},$$
as an explicit Gr\"obner basis computation for $K$ provides the leading 
monomials $(0,t^2,0)$ and $(0,0,t^3)$ arising from the s-polynomials of the
pairs $(sC_4,C_1)$ and $(s^2C_4,C_1)$. As we know that all ICMC2 surface 
singularities of type $(2,3,2)$ are adjacent to this singularitiy 
(see \cite{FN}), the principle of conservation of number then ensures the
upper semicontinuity of $\operatorname{dim}_{\mathbb C} H^1(Y,T_Y)$ concluding
the proof.
\end{proof}

\begin{rem}
The Milnor number of a semiquasihomogeneous ICIS is known to coincide 
with the Milnor number of its quasihomogeneous initial part, the Tjurina number
of a semiquasihomogeneous ICIS is bounded from above by the Tjurina number
of its quasihomogeneous initial part. Therefore the preceding lemma also 
establishes the inequality
$$ \mu \geq \tau -1$$
for semiquasihomogeneous isolated determinantal surface singularities of 
type $(2,3,2)$ with at most isolated singularities in the Tjurina transform.
\end{rem}

\begin{rem}
The considerations in the proof also show that $\operatorname{dim}_{\mathbb C} H^1(Y,T_Y)$ can be
computed explicitly in the non-quasihomogeneous case:\\
Considering the module $K$ defined above, we first observe that already 
$N'=(({\mathbb C}[s,t]\{\underline{x}\})^3/K)_{(1)}$ has to be
a finite dimensional vector space due to finite determinacy of the given singularity. 
Then the desired dimension of $H^1(Y,T_Y)$ is precisely
$\operatorname{dim}_{\mathbb C}(N') - \operatorname{dim}_{\mathbb C}(T)$
where $T=({\mathbb C}[s]\{\underline{x}\})^3/\overline{K}$. 
Note that the value of $\operatorname{dim}_{\mathbb C} H^1(Y,T_Y)$ can exceed $2$, as the following 
example shows:
$$\begin{pmatrix}
x_1 & x_2 & x_3 \cr
x_2^2+2x_4^2-x_3^3 & -x_1^2+x_3^3 &-x_2^2+x_4^2+x_1^3 
\end{pmatrix}$$
In this example, $\tau(X,0)=34$ and $\sum_{p\in {\operatorname{Sing}}(Y)} \tau(Y,p) =31$ which 
yields a difference of $\operatorname{dim}_{\mathbb C} H^1(Y,T_Y)=3$. However, $\mu(Y)=39$ and 
hence this does not provide a counterexample to Wahl's conjecture. 
\end{rem}

\begin{rem}
Looking at the above proof more closely, we can even determine the basis elements contributing to $N'/T$. Considering a standard basis
of $K$ w.r.t. the ordering chosen in the proof, they are precisely the monomials in the basis of $N'$
which are divisible by monomials in the leading module $L(\overline{K})$ 
arising from reduction of elements of the form
$$s^jx_4^i\begin{pmatrix}ta_{x_4} \cr tb_{x_4} \cr tc_{x_4} \end{pmatrix}$$
where $j \in \{1,2\}$ as we know from the proof and $i \in {\mathbb N}$ takes all values which are 
sufficiently low not to push the whole element beyond the determinacy bound.
\end{rem}

We now construct a counterexample to Wahl's conjecture, i.e. a
non-quasi\-homo\-geneous determinantal surface singularity for which 
$\mu = \tau -1$ holds. More precisely, we try to salvage as much of the 
situtation of lemma \ref{Wahl_ok} as we can: We search for an isolated surface
singularity $(X,0)$,
which gives rise to more than one isolated singularity in the Tjurina transform
$Y$. If we can choose all of these singularities as quasihomogeneous, but 
w.r.t. different weights for the respective singularities, and if we can 
furthermore ensure that $\operatorname{dim}_{\mathbb C} H^1(Y,T_Y)=2$, this 
is precisely the desired counterexample. All of these constraints are e.g. 
satisfied for the singularity $(X,0)$ defined by the maximal minors of the 
matrix
$$\begin{pmatrix} z+x & y & x^k+w^2 \cr
                  w^l & z & y \end{pmatrix},$$
where the value of $k,l \in {\mathbb N}$ is at least $3$. 
Its Tjurina transform $Y$ has two quasihomogeneous singularities, an $A_{l-1}$ 
at $(\underline{0},(0:1))$ and a $D_{k+1}$ at $(\underline{0},(1:0))$:
$$A_{l-1}: \langle w^l+tx+tz,z+ty,y+tx^k+tw^2 \rangle \sim_C 
  \langle w^l+tx+t^3x^k+t^3w^2,y,z \rangle$$
with monomial basis of the Tjurina algebra 
$$(1, 0, 0),(w, 0, 0), \dots ,(w^{l-2}, 0, 0)$$
and
$$D_{k+1}: \langle sw^l+z+x,sz+y,sy+x^k+w^2 \rangle \sim_C
  \langle z,y,s^2x+x^k+w^2+s^3w^l \rangle$$
with monomial basis 
$$ (0, 0, 1),
   (0, 0, s),
   (0, 0, s^2),
   (0, 0, x),\dots ,
   (0, 0, x^{k-2}).$$
As both of these are quasihomogeneous, the Tjurina and Milnor number 
coincide for each of the two singularities and we have a total Milnor 
number of of $Y$ of $k+l$ which implies that
$\mu(X,0) = k+l+1$. On the other hand, it is an easy computation to see that
a ${\mathbb C}$-vector space basis of $T^1(X)$ is given by the monomials:
$$\begin{pmatrix} 0 & 0 & 1 \cr 0 & 0 & 0 \end{pmatrix}, \dots ,
  \begin{pmatrix} 0 & 0 & x^{k-1} \cr 0 & 0 & 0 \end{pmatrix},
  \begin{pmatrix} 0 & 0 & 0 \cr 1 & 0 & 0 \end{pmatrix}, \dots ,
  \begin{pmatrix} 0 & 0 & 0 \cr w^{l-1} & 0 & 0 \end{pmatrix},$$
$$  \begin{pmatrix} 0 & 0 & 0 \cr 0 & 0 & 1 \end{pmatrix},
  \begin{pmatrix} 0 & 0 & 0 \cr 0 & 0 & w \end{pmatrix}. $$
Hence, the Tjurina number is $k+l+2$ and we can even discern the three 
contributions in the monomial basis: the first $k-1$ basis elements and the 
last $2$ correspond directly to a basis of the $T^1$ of the $D_{k+1}$ 
singularity\footnote{More precisely, the last $2$ elements correspond to those
basis elements involving $s$.}, the $(k+1)$-st to the $(k+l-1)$-th element 
to a basis of the $T^1$ of 
the $A_{l-1}$ singularity leaving precisely two elements for 
$H^1(Y,T_Y)$. \\
To see that the determinantal singularity is not quasihomogeneous, we consider
two hyperplane sections: with $V(w)$ and with $V(x)$. Both lead to 
quasihomogeneous space curve singularities, but w.r.t. different weights:
$$\begin{pmatrix} z+x & y & x^k \cr 0 & z & x \end{pmatrix} 
  {\textrm{ of Type }}  A_k \vee L {\textrm{ and }}$$
$$\begin{pmatrix} z & y & w^2 \cr w^l & z & y \end{pmatrix}
  {\textrm{ of Type }} E_{2l+6}(2) {\textrm{ (for }} l \not\equiv 2 \operatorname{mod} 3 {\textrm{) or }} J_{\frac{l+4}{6},0}(\frac{l+4}{6})$$
which are quasihomogeneous w.r.t.~weights $(2,k+1,2,-)$ in the first case and
$(-,l+4,2l+2,3)$ in the second. By an easy calculation, we see that it is 
impossible to choose weights for $y$ and $z$ satisfying both conditions at the
same time for $k,l \geq 3$.\\

We now come back to $3$-folds. The results of \cite{FZ} and in this article also allow a more geometric interpretation for 
the new and surprising phenomena observed in \cite{DP} for the simple ICMC2 $3$-fold 
singularities from \cite{FN}. Following Damon and Pike, we now consider the difference between 
the Euler characteristic of the Milnor fibre $b_3-b_2$ and the Tjurina number $\tau_X$ using 
their invariant
$\gamma := \tau - (b_3 - b_2):$

\begin{observation}[\cite{DP}]
\begin{itemize}
\item[a)] $\gamma \geq 2$ for all simple ICMC2-singularities of dimension $3$ and increases in 
value as we move higher in the classification.
\item[b)]  $b_3 - b_2 \geq -1$, with equality for the generic linear section and one infinite 
family.
\item[c)]  $b_3 - b_2$ is constant for certain infinite families with
values $-1$ (one family), $0$ (two families), and $1$ (two families).
\item[d)]  $\gamma$ is constant in all other considered infinite families in the table of simple 
singularities with only one exception where both $b_3 - b_2$ and $\gamma$ increase with $\tau$.
\item[e)]  For singularities of the form
$\begin{pmatrix}
x & y & z \\
w & v & g(x, y)
\end{pmatrix}$
with $g$ a simple hypersurface singularity, $\gamma = 3$ and $b_3 - b_2 = \mu(g) - 1$.
\end{itemize}
\end{observation}

To explain these observations, which differ greatly from the rather rigid structure observed 
in the surface case, we use again the Tjurina modification. In all cases in question, the
Tjurina transform has at most quasihomogeneous hypersurface singularities, whence we know that its
Milnor and Tjurina numbers coincide. From \cite{FZ}, we know that $b_2=1$ and $b_3$ coincides
with the Tjurina number of the Tjurina transform. So observation a) simply states 
that $\tau_X-\tau_Y \geq 1$. In particular, we have $\tau_Y=0$ and $\tau_X=1$ for the generic
linear section, the $A_0^+$-singularity, implying that this singularity is not adjacent to an 
$A_1$ singularity. As all other singularities in the list are adjacent to it, this explains 
the lower bound for the difference and hence observation a).\\

The first part of observation b) follows immediately from the fact that $b_2=1$. The second
part is concerned with the family
$$\begin{pmatrix}
x & y & z \cr
w & x & y+v^k
\end{pmatrix}.$$
The Tjurina transform of this family is smooth, which implies that $b_3=0$. The Tjurina number of
$X$, however, increases in the family, $\tau_X=2k-1$, and is closely related to the maximal 
number of $A_0^+$-singularities which can appear in a fibre of the versal family. This maximal 
number is achieved e.g. by the perturbation
$$\begin{pmatrix}
x & y & z \cr
w & x & y+v^k+\alpha
\end{pmatrix}$$
for any $0 \neq \alpha \in {\mathbb C}$, where we see precisely $k$ such singularities. 
Observation c) then simply states that similar behaviour with constant topological type of the
Tjurina transform also occurs for other families which have singularities in their Tjurina 
transform.\\

Observation d), on the other hand, singles out families in which the increase of Tjurina number
originates from the increase in Milnor/Tjurina number of the Tjurina transform and the maximal
number of $A_0^+$-singularities appearing in a fibre of the versal family does not change. In
the last considered family, where $b_3-b_2$ and $\gamma$ increase with $\tau_X$, we see a first
example of increasing contributions to both the Tjurina transform and the purely determinantal
part.\\

The only part of the last observation, which still remains to be explained, is the statement 
$\gamma=3$. As already observed in \cite{FN}, $T^1_X$ is isomorphic to $T^1_{\rm section}$
of the plane curve singularity defined by the right hand lower entry. Hence $\tau_X - \tau_Y$
is the difference arising from deformations with section as opposed to usual deformations 
for the respective plane curve. In all cases in question, this difference is precisely $2$ giving
rise to $\gamma=(\tau_X-\tau_Y)+b_2=3$.

\begin{rem}
Although this article explains many of the recent surprising observations 
about ICMC2 singularities, this is merely a glimpse into the new phenomena 
we are seeing in determinantal singularities. Extending exisiting tools to
the determinantal setting and combining methods from the theory of syzygies,
from classical singularity theory and topology, we seem to have reached a 
point now, where we can start thinking about a more systematic study of 
general determinantal singularities. 
\end{rem}

\end{document}